\documentclass[12pt]{amsart}
\usepackage[T1]{fontenc}
\usepackage{amssymb}
\renewcommand{\int}{\operatorname{Int}}

\newtheorem{theorem}{Theorem}
\newtheorem{cor}[theorem]{Corollary}
\newtheorem{lem}[theorem]{Lemma}
\newtheorem{pro}[theorem]{Proposition}

\author{Piotr Kalemba}
\address{Piotr Kalemba \\
 Institute of Mathematics, University of
Silesia \\
ul. Bankowa 14, 40-007 Katowice}
\email{piotr.kalemba@us.edu.pl}
\author{Szymon Plewik}
\address{Szymon Plewik\\Institute of Mathematics,
University of Silesia, ul. Ban\-ko\-wa 14, 40-007 Katowice}
\email{plewik@math.us.edu.pl}

\begin{document}

\title{On  regular but not completely regular spaces}
\date{}
\maketitle
\begin{abstract}
We present how  to obtain non-comparable regular but not completely regular spaces. We analyze a generalization of Mysior's example, extracting  its underlying purely set-theoretic  framework.  
This enables us to build simple counterexamples, using the Niemytzki plane, the Songefrey plane or  Lusin gaps.
\end{abstract}

\section{Introduction}\label{ss1}
  Our discussion  focuses around a question: \textit{How can a completely regular space be extended by a point to only a regular space?} Before A. Mysior's  example,  such a construction seemed quite complicated, compare  \cite{mys}  and \cite{en}. R. Engelking included a description of Mysior's example in the Polish edition of his book \cite[p. 55-56]{eng}. In \cite{ciw} there is considered  a modification of  Mysior's example which  requires no algebraic structure on the space.    We present a purely set-theoretic approach  which enables us to obtain non-comparable examples, such spaces are   $X(\omega, \lambda_1)$ and $X( \lambda_2, \kappa)$, see Section \ref{s1}. This approach is a step towards  a procedure to rearrange   some completely regular spaces   onto only regular ones. One can find a somewhat similar idea in \cite{jon}, compare "the Jones' counterexample machine" in \cite[p. 317]{ciw}. The starting point of our discussion are the cases of  completely regular spaces which are not normal.   For example, subspaces of the Niemytzki plane are examined in  \cite{cho} or \cite{ss},  some $\Psi$-spaces are studied in \cite{hh},  also  the Songefrey  plane is commentated in \cite{ss}. 
The key idea of our construction of counterexamples looks roughly as follows. Start from a completely regular space $X$,  which is not normal. In fact, we need that $ X $  contains countable many pairwise disjoint closed subsets which,  even after removal from each of them a small subset, cannot be separated by open sets.  By numbering these closed sets as $\Delta_X(k)$ and assuming that the collections of small sets form proper ideals $I_X(k)$, we should check that   the property $(*)$ is fulfilled. Copies of  $X$ are numbered by integers and then the $k$-th copy is glued along the set  $\Delta_X(k)$ to  the $(k-1)$-copy, moreover copies of sets $\Delta_X(m)$, for $k\not=n\not=k-1$, are removed from the   $k$-th copy.  As a result we get the completely regular space $ \textbf{Y}_X $, which has a one-point extension to the regular space which is not completely regular. 

In fact, given a completely regular space X, which we do not know whether it has one-point extension to the space which is only regular, we can build a space $\textbf{Y}_X$ which has such an extension. A somewhat similar method was presented in \cite{jon}.  For this reason, we  look for ways of comparing  such spaces.
Following the concept of topological ranks,   compare \cite[p. 112]{kur} or \cite[p. 24]{sie}, which was developed in Polish School of Mathematics,  we say that spaces $X$ and $Y$ have non-comparable \textit{regularity ranks},  whenever $X$  and $Y$ are regular but not completely regular and there does not exist a regular but not completely regular space $Z$ such that $Z$ is homeomorphic to a subspace of $X$ and  $Z$ is homeomorphic to a subspace of  $Y$.

\section{On  Mysior's example}\label{s1}
We modify the approach carried out in \cite{ciw}, which
 consists in a generalization  of Mysior's example,  compare \cite{mys}.  Despite the fact that our arguments resemble those used in \cite{ciw}, we believe that  this presentation is a bit simpler 
and enables us to construct some non-homeomorphic examples, for example  spaces $X(\lambda, \kappa)$.  
  Let $\kappa$ be an uncountable cardinal and $\{A(k): k \in \mathbb Z \}$ be a countable  infinite partition of $\kappa$ into  pairwise disjoint subsets of the cardinality $\kappa$, where $\mathbb Z$ stands for the integers.  
 Denote   the diagonal of the Cartesian product  $\kappa^2$  by $\Delta = \{(x,x): x \in \kappa \}$ and put $\Delta(k)= \Delta \cap A(k)^2$.

Fix an infinite cardinal number $\lambda < \kappa$ and  proper $\lambda^+$-complete ideals $I(k,\lambda)$ on the sets $A(k)$. 
In particular,  we assume that singletons are in $ I(k,\lambda )$, hence   $ H \in I(k,\lambda) $ for any  $H \subseteq A(k)$ such that  $|H| \leqslant  \lambda$.
Consider a topology $\mathcal T$ on $X=\kappa^2$ generated by the basis consisting of all singletons $\{ a \}$, whenever $a \in \kappa^2 \setminus \Delta $, and all the sets $$ \{(x,x)\}\cup (\{x\}\times (A(k-1)\setminus G)) \cup ((A(k+1) \setminus F)\times \{x\}) ,$$ where $x\in A(k)$  and   $| G| < \lambda $ and $F\in I(k+1, \lambda)$. We denote  not-singleton basic sets by 
$\Gamma(x, G,F)$.

\begin{lem}\label{1} Assume that $H \subseteq \Delta (k)  \cap V,$ where $V$ is an open set  in $X$. If the set   $ \{x\in A(k): (x,x) \in H\}$ does not belong to the ideal $ I(k, \lambda)$, then the difference  $\Delta(k-1)\setminus cl_X (V)$ has the cardinality less than $\lambda$.
 \end{lem}
\begin{proof} Suppose that  a set $\{(b_\alpha,b_\alpha): \alpha < \lambda \} \subseteq \Delta(k-1)$ of the cardinality $\lambda$ is disjoint from $ cl_X(V)$. For each $\alpha < \lambda$, fix a basic set  
$\Gamma(b_\alpha, G_\alpha, F_\alpha)$ disjoint from $V$, where $F_\alpha \in I(k, \lambda)$. The ideal  $I(k, \lambda)$ is  $\lambda^+$-complete and the set $\{x\in A(k): (x,x) \in H\} $ does not belong to this ideal. So, there exists  a point  $(x,x) \in H$ such that $$x\in A(k)\setminus \bigcup \{ F_\alpha : \alpha < \lambda \}.$$  Therefore  $$(x, b_\alpha) \in (A(k)\setminus F_\alpha) \times \{b_\alpha\} \subseteq  \Gamma(b_\alpha, G_\alpha, F_\alpha)$$ for every $\alpha < \lambda$. Fix  a basic set $\Gamma(x, G_x, F_x) \subseteq V$ and $\alpha < \lambda$ such that 
$b_\alpha \in A(k-1)\setminus G_x$. We get  $(x, b_\alpha) \in \{x\}\times (A(k-1)\setminus G_x) \subseteq V $,  a contradiction.  \end{proof}

\begin{cor} The space $X$  is completely regular, but not normal. \end{cor}
\begin{proof} The base consists of closed-open sets and one-points subsets of $X$ are closed. So $X$, being  a zero-dimensional T$_1$   space, is completely regular.  
Subsets $\Delta(k+1)$ and $\Delta (k)$ are closed and disjoint. By Lemma \ref{1}, if  a set $V\subseteq X$ is  open and $ \Delta(k +1) \subseteq V$, then  
$cl_X(V) \cap \Delta(k)\not=\emptyset$, which implies  that $X$ is not a normal space. \end{proof}

\begin{pro}\label{3} Assume that  the cardinal $\lambda$ has an uncountable cofinality. If  $f:X\to \mathbb R$ is a continuous real valued function, then for  any point $x \in \kappa$ there exists a basic set $\Gamma(x, G_x, F_x)$ such that the function $f$ is constant on it.    \end{pro}  
\begin{proof} Without loss of generality, we can assume that $f(x,x)=0$.  For each $n>0$, fix a base set 
$\Gamma(x, G_n,F_n) \subseteq f^{-1}((\frac{-1}{n},\frac1n))$. Then put $G_x= \cup \{G_n: n>0\}$ and $F_x= \cup \{F_n: n>0\}$. Since $\lambda$ has an uncountable cofinality,  we get that the set  $ \Gamma(x, G_x,F_x)$ belongs to the base. Obviously, if $$\textstyle (a,b) \in \cap \{\Gamma(x, G_n, F_n): n>0\}= \Gamma(x, G_x,F_x) ,$$ then $f(a,b)=0$. \end{proof}

When $\lambda$ has the  countable cofinality, then the above proof also works, but then the set $G_x = \cup \{G_n: n>0\} $ may have the cardinality $\lambda$, and therefore $\Gamma(x, G_x,F_x)$  does not necessarily belong to the base, and also it could be not open. Furthermore, any continuous real valued function must also be constant onto other large  subsets of X.

\begin{lem}\label{l2} Let $k \in \mathbb Z$. If  $f:X\to \mathbb R$ is a continuous real valued function, then for any $\varepsilon >0$ there   exists a real number $a$ such that  $f[\Delta(k-1)] \subseteq[a,a+3\cdot \varepsilon]$ for all but less than $\lambda$ many  points $(x,x) \in \Delta(k-1)$.  \end{lem}
\begin{proof} Fix a real number $b$ and $\varepsilon >0$. The ideal $I(k, \lambda)$ is $\lambda^+$-complete, so we can choose  an integer $q \in \mathbb Z$ such that the subset  $$ \{ x\in A(k): f(x,x) \in[b+q \cdot \varepsilon, b+(q+1) \cdot \varepsilon ]\}$$ does not belong to  $I(k, \lambda)$. Use Lemma \ref{1}, putting  $a= b+(q-1) \cdot \varepsilon$ and $H= f^{-1}([a+\varepsilon, a+2\cdot \varepsilon ]) \cap \Delta (k)$ and $V=f^{-1}((a, a+3\cdot \varepsilon ))$. Since $cl_X (V) \subseteq f^{-1}([a, a+3\cdot \varepsilon ])$, the proof is completed.  
\end{proof}
 
\begin{cor}\label{c3}
If  $f:X\to \mathbb R$ is a continuous real valued function, then for any $k \in \mathbb Z$ there   exists a real number $a_k$ such that  $f(x,x) = a_k$  for all but  $\lambda$ many  points $(x,x) \in \Delta(k)$. Moreover, if $\lambda$ has uncountable cofinality, then  $f(x,x) = a_k$  for all  but less than $\lambda$ many   $x \in A(k)$.
 \end{cor}
\begin{proof} Apply Lemma \ref{l2}, substituting consecutively $ \frac1n$ for  $\varepsilon $, for $n>0$,  and $k+1$ for $k$.   \end{proof}

\begin{theorem} \label{6} If $f:X\to \mathbb R$ is a continuous real valued function, then there exists $a\in \mathbb R$ such that $f(x,x)=a$ for all but  $\lambda$ many $x\in \kappa$.
Moreover, when $\lambda$ has an uncountable cofinality, then $f(x,x)=a$ for all but less than $\lambda$ many $x\in \kappa$.  
\end{theorem}
\begin{proof}
We shall to prove that the numbers $a_k$  which appear in Corollary \ref{c3} are equal. To do this, suppose that $a_k \not= a_{k-1}$ for some $k\in \mathbb Z$. Choose disjoint open  intervals $\mathbb J$ and $\mathbb I$ such that  $ a_k \in \mathbb J$ and $a_{k-1}\in \mathbb I$. Apply Lemma \ref{1}, taking $ H=\{(x,x) \in \Delta(k): f(x,x) = a_k \}$ and $V= f^{-1} (\mathbb J)$.    Since $cl_X(V) \cap f^{-1} (\mathbb I) =\emptyset$,  we get  $f(x,x) \not= a_{k-1} $ for all but less than $\lambda$ many $x\in A(k-1)$, a contradiction. \end{proof}

Knowing  infinite cardinal numbers $\lambda < \kappa$ and  proper $\lambda^+$-complete ideals $I(k,\lambda)$ on sets $A(k)$,  one can 
extend the space $X$  by one  or two points  so as to get a regular space which is not completely  regular. This is a standard construction, compare \cite{jon}, \cite{mys} and \cite{ciw} or \cite[Example 1.5.9]{eng}, so we will describe it briefly. Fix  points $+\infty$ and $-\infty$ that do not belong to $X$. On the set $X^*= X \cup\{-\infty, +\infty\}$ we introduce the following topology. Let open sets in $X$ be open in $X^*$, too. But  the sets  $$\mathcal V^+_m=\{+\infty\} \cup \bigcup\{A(n) \times \kappa: n>m\}$$  form a base  at the point  $ +\infty $ and  the sets  $$\mathcal V^-_m=\{-\infty\} \cup \bigcup\{A(n) \times \kappa: n\leqslant m\}\setminus \Delta (m)$$  form a base  at the point  $ -\infty .$ Thus we have 
$$\Delta (m) = cl_{X^*}(\mathcal V^+_m) \cap cl_{X^*}(\mathcal V^-_m) = cl_{X}(\mathcal V^+_m \cap X) \cap cl_{X}(\mathcal V^-_m \cap X),$$ which gives  that the space $X^*$ is regular and not completely regular. Indeed, consider a closed subset $D \subseteq X^*$  and a point $p\in X^* \setminus D$.  When $p\in X$, then $p$ has a closed-open neighborhood in $X^*$ which is  disjoint with $D$.  When $p=+\infty$, then    consider the basic set $\mathcal V^+_m$ which is disjoint with $D$ and check $cl_{X^*}(\mathcal V^+_{m+1}) \subseteq \mathcal V^+_m. $ Analogously, 
 when $p=-\infty$, then    consider the basic set $\mathcal V^-_m$ which is disjoint with $D$ and check $cl_{X^*}(\mathcal V^-_{m-1}) \subseteq \mathcal V^-_m. $
By Theorem \ref{6}, no continuous real valued function separates an arbitrary closed set $\Delta(k)$  from a point $p\in \{+\infty, -\infty\}$.  Hence the space $X^*$ is not completely regular. The same holds for subspaces   $X^*\setminus \{+\infty\}$ and $X^*\setminus \{-\infty\}$. Moreover, if $f: X^* \to \mathbb R$ is a continuous function, then $f(+\infty) = f(-\infty).$

Now for convenience, the above defined  space $X$ is denoted  $ X(\lambda, \kappa)$,  whenever the ideals $ I(\lambda, \kappa)$ consist of sets of the cardinality less than $\lambda$. Assuming $\omega < \lambda_1 < \lambda_2 < \kappa$ we get two (non-comparable) non-homeomorphic spaces   $X(\omega, \lambda_1)$ and $X( \lambda_2, \kappa)$, since the first one has the cardinality $\lambda_1.$ But  a subspace of $X( \lambda_2, \kappa)$ of the cardinality $\lambda_1$ is discrete and its closure in $(X(\omega, \lambda_2))^*$, being zero-dimensional, is completely regular. In other words, spaces $(X(\omega, \lambda_1))^*$ and $(X( \lambda_2, \kappa))^*$ have non-comparable regularity ranks.

\section{General approach}  The  analysis conducted above can be generalized using some known counterexamples. We apply such a generalization  to  the Niemytzki plane, cf. \cite[p. 34]{eng} or \cite[pp. 100 - 102]{ss}, the Songefrey's half-open square topology, cf. \cite[pp. 103 - 105]{ss} and special  Isbell-Mrówka spaces (which are also known as $\Psi$-spaces).  

Given a  space $X$ and a closed and discrete subset $\Delta_X\subseteq X$, assume  that $\Delta_X$ can be partitioned onto   pairwise disjoint subsets $\Delta_X(k)$. For each $k\in \mathbb Z$, let  $I_X(k)$ be a proper ideal on $\Delta_X(k)$.  Suppose that the following property is fulfilled:

($*$).  \textit{If a set $V \subseteq X$ is open and the set $\Delta_X(k) \setminus  V$ belongs to $  I_X(k)$, then the set $\Delta_X(k-1) \setminus cl_X(V)$ belongs to $I_X(k-1)$.} 

\noindent Then it is possible to give a general scheme of a construction of a completely regular space $\textbf{Y}=\textbf{Y}_X$, which has one-point extension to a regular space which is not completely regular and two-point extension  to a regular space such that no continuous real valued function separates the extra points, whereas removing  a single point we get a regular space which is not completely regular. 
To get this we put  $$\textbf{x}_k = \begin{cases}
                        (k,x),  \text{ when } x\in X\setminus \Delta_X;\\
                       \{(k,x), (k+1,x)\},  \text{ when  } x\in  \Delta_X(k).\\
                    \end{cases}$$
And then put  $\textbf{Y}_X=  \{\textbf{x}_k: x\in X \mbox{ and } k \in \mathbb Z \}.$
  Endow $\textbf{Y}_X$ with the  topology as follows.
	If $k \in \mathbb Z $ and $ V\subseteq X\setminus \Delta_X$    is an open subset of $X$, then the set $\{\textbf{x}_k: x \in  V\}$ is  open in $\textbf{Y}_X.$ 
	Thus we define neighborhoods of the point $\textbf{x}_k$ where $x \notin \Delta_X$. To define  neighborhoods of the point $\textbf{x}_k$,  where $x \in \Delta_X$, we  use the formula:
If $k \in \mathbb Z$ and   $V \subseteq X$ is an open subset,   then the set 
	$$\{\textbf{x}_k: x \in  V\}\cup \{\textbf{x}_{k+1}: x \in  V\setminus \Delta_X\}$$ is  open  in $\textbf{Y}_X.$ To get a version of $(*)$,
	we  put the following: 
 $\Delta_\textbf{Y}(k) = \{\textbf{x}_k: x \in \Delta_X(k)\}$; 
 $\Delta_\textbf{Y}=\bigcup\{\Delta_X(k): k\in \mathbb Z\}$; 
	 Let $I_\textbf{Y}(k)$ be a proper ideal which consists of sets $\{\textbf{x}_k: x \in A\}$ for $A  \in I_X(k)$;
$\textbf{Y}_k=  \{\textbf{y}_k: y \in X\setminus \Delta_X\}.$
	So,  if $k\in \mathbb Z$,  then $$\Delta_\textbf{Y}(k) = cl_\textbf{Y}(\{\textbf{y}_k: y \in X\setminus \Delta_X\})\cap cl_\textbf{Y}(\{\textbf{y}_{k+1}: y \in X\setminus \Delta_X\}).$$ 
As we can see, the properties of the space $\textbf{Y}_X$ can be automatically rewritten from the relevant properties of $ X $, so we leave  details to the reader.
	
	\begin{pro}\label{p7} Assume that a space $X$ satisfied $(*)$ and the space $\textbf{Y}$ is as above.  If a set $\textbf{V} \subseteq \textbf{Y}$ is open and the set $\Delta_\textbf{Y} (k) \setminus  \textbf{V}$ belongs to $  I_\textbf{Y}(k)$, then the set $\Delta_\textbf{Y} (k-1) \setminus cl_\textbf{Y}(\textbf{V})$ belongs to $I_\textbf{Y}(k-1)$.  \hfill $\Box$
	\end{pro} 
	\begin{pro} If a space $X$ is completely regular, then the space $\textbf{Y}$ is completely regular, too.  \hfill $\Box$ \end{pro}

	Now, fix  points $+\infty$ and $-\infty$ that do not belong to $\textbf{Y}$. On the set $\textbf{Y}^*= \textbf{Y} \cup\{-\infty, +\infty\}$ we introduce the following topology. Let open sets in $\textbf{Y}$ be open in $\textbf{Y}^*$, too. But  the sets  $$\mathcal V^+_m=\{+\infty\} \cup \bigcup\{\textbf{Y}_n: n\geqslant m\}\cup \bigcup \{\Delta_\textbf{Y}(n): n>m\}$$  form a base  at the point  $ +\infty $ and  the sets  $$\mathcal V^-_m=\{-\infty\} \cup \bigcup\{\textbf{Y}_n: n\leqslant m\}\cup \bigcup \{\Delta_\textbf{Y}(n): n<m\}$$  form a base  at the point  $ -\infty .$ Thus we have 
	$$\Delta_\textbf{Y} (m) \subseteq  cl_{\textbf{Y}^*}(\mathcal V^+_m) \cap cl_{\textbf{Y}^*}(\mathcal V^-_m)= \Delta_\textbf{Y} (m) \cup \textbf{Y}_m,$$ which implies  the following. 
	\begin{theorem}\label{9} If $f:\textbf{Y}^*\to \mathbb R$ is a continuous real valued function, then $f(+\infty)=f(-\infty)$.   
	\end{theorem}
	\begin{proof} Suppose  $f: \textbf{Y}^* \to \mathbb R$ is a continuous function such that  $f(+\infty) = 1$ and  $f(-\infty) = 0.$ Fix a decreasing sequence $\{\epsilon_n\}$ which converges to $\frac12$.
	 Thus $$f^{-1}((\epsilon_n, 1]) \subseteq cl_{\textbf{Y}^*}(f^{-1}((\epsilon_n, 1])) \subseteq f^{-1}([\epsilon_n, 1]) \subseteq f^{-1}((\epsilon_{n+1}, 1]).$$ By Proposition \ref{p7}, if $K_m\in I_\textbf{Y}(m)$ and $\Delta_\textbf{Y}(m) \setminus K_m \subseteq f^{-1}((\epsilon_n, 1]),$ then 
	$$   f^{-1}((\epsilon_{n+1}, 1]) \supseteq \Delta_\textbf{Y}(m-1) \setminus K_{m-1}, $$
		for  some  $K_{m-1}\in I_\textbf{Y}(m-1).$  Since there exists $m \in \mathbb Z$ such that $+\infty \in \mathcal V^+_m \subseteq f^{-1}((\epsilon_0, 1]) $, inductively,  we get $$\Delta_{\textbf{Y}}\setminus \bigcup\{K_n: n\in \mathbb Z \} \subseteq f^{-1}([\frac12, 1]), $$  which implies that each $ \mathcal V^-_n$ contains a point $\textbf{y}\in \textbf{Y} $ such that $f(\textbf{y})\geqslant \frac12.$ Hence $f(-\infty) \geqslant \frac12$, a contradiction.    \end{proof}

	\subsection{Application of the Niemytzki plane}
	
	Recall that the Niemytzki plane $\mathbb P = \{(a,b) \in \mathbb R \times \mathbb R: 0\leqslant b \}$  is the closed half-plane which is endowed with the topology generated by  open discs disjoint with the real axis $\Delta_{\mathbb P}=\{(x,0): x \in \mathbb R\}$ and all sets of the form $\{a\} \cup D$ where $D\subseteq \mathbb P$ is an open disc which is tangent to $\Delta_{\mathbb P}$ at the point $a \in \Delta_{\mathbb P}$.
Choose  pairwise disjoint subsets $\Delta_{\mathbb P} (k)\subseteq \Delta_{\mathbb P}$, where $k\in \mathbb Z$, such that each set $\Delta_{\mathbb P} (k)$ meets every dense $G_\delta$ subset of the real axis. 
	To do that is enough to slightly modify the classic construction of a Bernstein set. Namely, fix an enumeration $\{A_\alpha: \alpha < \frak{c}\}$	of all dense $G_\delta$ subsets of the real axis. Defining inductively at step $\alpha$ choose  a (1-1)-numerated subset $\{p_k^\alpha : k \in \mathbb Z\}\subseteq A_\alpha \setminus \{p_k^\beta :  k \in \mathbb Z \mbox{ and } \beta < \alpha \}$.  Then,   for each   $k \in \mathbb Z, $ put $\Delta_{\mathbb P} (k) = \{p^\alpha_k: \alpha < \frak{c} \}.$
	
Let us assume that if $F \subseteq \mathbb R \times \mathbb R$, then the topology on $F$ induced  from the Euclidean topology will be  called the \textit{natural topology} on $F$. A set, which is a countable union of nowhere dense subsets  in the natural topology on  $F $,  will be called \textit{a set of first  category} in $F$.   Our proof of the following lemma is a modification of known reasoning justifying that $\mathbb P$ is not a normal space, compare \cite[pp. 101 -102]{ss}.
	\begin{lem}\label{7}  Let a set $F \subseteq \Delta_{\mathbb P}$ be a dense subset in the natural topology on the real axis $\Delta_{\mathbb P}$.  If a set $V$ is open in $\mathbb P$ and $F \subseteq V$, then the set $\Delta_{\mathbb P} \setminus cl_{\mathbb P}(V)$ is of first category in $\Delta_{\mathbb P}$.   \end{lem}
	\begin{proof} To each point $a\in \Delta_{\mathbb P}\setminus cl_{\mathbb P}(V)$ there corresponds a disc $D_a \subseteq \mathbb P \setminus cl_{\mathbb P}(V)$ of radius $r_a$ tangent to $\Delta_{\mathbb P}$ at the point $a$. Put $$S_n=\{
	a\in \Delta_{\mathbb P} \setminus cl_{\mathbb P}(V): r_a \geqslant \frac1n \}  $$
and  use  density of  $F$ to check that each $S_n$ is nowhere dense in the natural topology on $\Delta_{\mathbb P}$. So $\bigcup \{S_n: n>0\}=\Delta_{\mathbb P} \setminus cl_{\mathbb P}(V).$ \end{proof}

 The space $\textbf{Y}_{\mathbb P}$ is completely regular. The subspaces $\textbf{Y}_{\mathbb P}\cup \{-\infty\}$, $\textbf{Y}_{\mathbb P}\cup \{+\infty\}$ and and the space $\textbf{Y}_{\mathbb P}^*$ are regular. Moreover, if $f:\textbf{Y}^*_{
\mathbb P}\to \mathbb R$ is a continuous real valued function, then $f(+\infty)=f(-\infty)$.

\subsection{Application of the Songefrey  plane,} i.e. application of the Songenfrey's half-open  square topology. 
 
Recall that the Songefrey plane $\mathbb S = \{(a,b):  a\in \mathbb R \mbox{ and } b  \in \mathbb R  \}$  is the plane  endowed with the topology generated by  rectangles of the form 
$[a,b) \times [c,d)$. Let $\Delta_{\mathbb S}=\{(x,-x): x \in \mathbb R\}$. 
Since $\Delta_{\mathbb S}$ with the topology  induced  from the Euclidean topology is homeomorphic with the real line, we can choose  pairwise disjoint subsets $\Delta_{\mathbb S} (k)\subseteq \Delta_{\mathbb S}$  
such that each set $\Delta_{\mathbb S} (k)$ meets every dense $G_\delta$ subset of $\Delta_{\mathbb S}$. 
The following lemma can be proved be the second category argument used previously in the proof Lemma \ref{7}, so we omit it, compare also  \cite[pp. 103 -104]{ss}.
	\begin{lem}\label{8}  Let a set $F \subseteq \Delta_{\mathbb S}$ be a dense subset in the topology on $\Delta_{\mathbb S}$ which is inherited from the Euclidean topology.  If a set $V$ is open in $\mathbb S$ and $F \subseteq V$, then the set $\Delta_{\mathbb S} \setminus cl_{\mathbb S}(V)$ is of first category in $\Delta_{\mathbb S}$.   \hfill $\Box$ \end{lem}
	
Again,
 the space $\textbf{Y}_{\mathbb S}$ is completely regular. The subspaces $\textbf{Y}_{\mathbb S}\cup \{-\infty\}$, $\textbf{Y}_{\mathbb S}\cup \{+\infty\}$ and and the space $\textbf{Y}_{\mathbb S}^*$ are regular. Moreover, if $f:\textbf{Y}^*_{
\mathbb S}\to \mathbb R$ is a continuous real valued function, then $f(+\infty)=f(-\infty)$.

\subsection{Applications of some $\Psi$-spaces} 

Let us recall some notions needed to define a Lusin gap, compare \cite{luz}. A family of sets is called \textit{almost disjoint}, whenever any
two members of it  have the finite intersection.
A set $C$ \textit{separates} two families, whenever each member
of the first family  is almost contained in $C$, i.e. $B \setminus  C$  is finite for any $B \in Q$, and
each member of the other family is almost disjoint with $C$.  An uncountable  family $\mathcal L$, which consists of almost disjoint and infinite subsets of $\omega$, 
  is called \textit{Lusin-gap,} whenever
 no two its uncountable and disjoint   subfamilies can be separated by a subset of $\omega$.  Adapting concepts discussed in \cite{mro} or \cite{hh},   to a Lusin-gap $\mathcal L$,  let $ \Psi(\mathcal L) = \mathcal L \cup \omega$. A topology on $\Psi(\mathcal L)$ is generated as  follows. Any subset of $\omega$   is open, also  for each point $A\in \mathcal L$  the sets $\{ A\} \cup A \setminus F$, where $F$ is  finite, are open.   

\begin{pro} If $\mathcal L$ is a Lusin-gap and  $\bigcup \{ \Delta_{\mathcal L}(k): k\in \mathbb Z \} = \mathcal L$,  then the space  $\Psi(\mathcal L)$ satisfies the property $(*)$, whenever  sets $\Delta_\mathcal L(k) $ are uncountable and pairwise disjoint and each ideal $I_{\mathcal L}(k)$ consists of all countable subsets of $\Delta_{\mathcal L}(k)$.   
\end{pro}
\begin{proof}
 Consider     uncountable and disjoint families  $\mathcal{  A, B} \subseteq \mathcal L$.  Suppose $\mathcal A \subseteq V$ and    $\mathcal B \subseteq W,$ where open sets $V$ and $ W$ are disjoint.   Let $$ C= \bigcup \{A \subseteq \omega: \{A\} \cup A  \mbox{ is almost contained in   } V \}.  $$ The set $C$ separates families $\mathcal{  A}$ and $\mathcal  B$, which  contradicts that  $\mathcal L$ is a Lusin-gap. Setting $\mathcal{  A} = \Delta_{ \mathcal L} (k)$ and $\mathcal{  B} = \Delta_{ \mathcal L}(k-1)$, we are done.
 \end{proof} 

 The space $\textbf{Y}_{\mathcal L}$ is completely regular. Again by Theorem \ref{9}, we get the following. 
 The subspaces $\textbf{Y}_{\mathcal L}\cup \{-\infty\}$, $\textbf{Y}_{\mathcal L}\cup \{+\infty\}$ and and the space $\textbf{Y}_{\mathcal L}^*$ are regular. Moreover, if $f:\textbf{Y}^*_{
\mathcal L}\to \mathbb R$ is a continuous real valued function, then $f(+\infty)=f(-\infty)$.

\section{Comment}

In \cite{jon}, F. B. Jones formulated the following problem: \textit{Does a non-completely regular space always contain a substructure similar to that possessed by $Y$}? Jones' space $Y$
is constructed by gluing (sewing) countably many disjoint copies of a suitable space $X$. This method fixes two subsets of $X$ and consists in sewing alternately  copies of either of them. On the other hand, our method  consists in gluing different sets  at each step.   The problem of Jones may be understood as an incentive to study the structural diversity of regular spaces, which are not completely regular. Even though the meaning of  "a substructure similar to that possessed by $Y$" seems vague, we think that  an appropriate criterion  for the aforementioned diversity is a slightly modified concept of a topological rank, compare \cite{kur} or \cite{sie}.  We have introduced regularity ranks, but our counterexamples  are only a preliminary step to the study of  diversity of regular spaces.

\end{document}